\documentclass[12pt]{amsart}
\usepackage{enumerate}
\usepackage{hyperref}
\usepackage{graphicx,color}
\usepackage{cite}
\usepackage{cancel}
\usepackage{amsthm, amsmath, amssymb, mathtools}
\usepackage[top=45truemm, bottom=45truemm, left=30truemm, right=30truemm]{geometry}

\renewcommand{\epsilon}{\varepsilon}
\renewcommand{\limsup}{\varlimsup}
\renewcommand{\liminf}{\varliminf}

\newcommand{\abs}[1]{\left|#1\right|}

\newcounter{count}
\newcommand{\num}{\stepcounter{count}\the\value{count}}

\numberwithin{equation}{section}

\newtheorem{theorem}{Theorem}[section]
\newtheorem{lemma}[theorem]{Lemma}
\newtheorem{corollary}[theorem]{Corollary}

\newtheorem{proposition}[theorem]{Proposition}

\theoremstyle{definition}
\newtheorem{remark}[theorem]{Remark}

\newcommand{\cC}{\mathcal{C}}

\newcommand{\cE}{\mathcal{E}}

\newcommand{\cI}{\mathcal{I}}

\newcommand{\cN}{\mathcal{N}}

\newcommand{\cR}{\mathcal{R}}
\newcommand{\cS}{\mathcal{S}}

\newcommand{\PS}{\mathrm{PS}}

\title[Representations of integers as differences between PS numbers]{On the number of representations of integers as differences between Piatetski-Shapiro numbers}

\author[Y.\ Yoshida]{Yuuya Yoshida}
\address{Yuuya Yoshida\\
Nagoya Institute of Technology\\ Gokiso-cho\\ Showa-ku\\ Nagoya\\ 466-8555\\ Japan}
\curraddr{}
\email{yyoshida9130@gmail.com}



\begin{document}
\maketitle

\begin{abstract}
For $\alpha>1$, set $\beta=1/(\alpha-1)$. We show that, for every $1<\alpha<(\sqrt{21}+4)/5\approx1.717$, the number of pairs $(m,n)$ of positive integers with $d=\lfloor{n^\alpha}\rfloor - \lfloor{m^\alpha}\rfloor$ is equal to $\beta\alpha^{-\beta}\zeta(\beta)d^{\beta-1} + o(d^{\beta-1})$ as $d\to\infty$, where $\zeta$ denotes the Riemann zeta function. We use this result to derive an asymptotic formula for the number of triplets $(l,m,n)$ of positive integers such that $l<x$ and $\lfloor{l^\alpha}\rfloor + \lfloor{m^\alpha}\rfloor = \lfloor{n^\alpha}\rfloor$. Furthermore, we prove that the additive energy of the sequence $(\lfloor{n^\alpha}\rfloor)_{n=1}^N$, i.e., the number of quadruples $(n_1,n_2,n_3,n_4)$ of positive integers with $\lfloor{n_1^\alpha}\rfloor+\lfloor{n_2^\alpha}\rfloor=\lfloor{n_3^\alpha}\rfloor+\lfloor{n_4^\alpha}\rfloor$ and $n_1,n_2,n_3,n_4\le N$, is equal to $O_\alpha(N^{4-\alpha})$ when $1<\alpha\le4/3$.
\end{abstract}

\section{Introduction}
\subsection{The number of solutions of the equation $x+y=z$ in $\PS(\alpha)$}

Let $\PS(\alpha)$ be the set $\{ \lfloor{n^\alpha}\rfloor : n\in\mathbb{N} \}$ for $\alpha\ge1$, 
where $\lfloor{x}\rfloor$ (resp.\ $\lceil{x}\rceil$) denotes 
the greatest (resp.\ least) integer $\le x$ (resp.\ $\ge x$) for a real number $x$, 
and $\mathbb{N}$ denotes the set of all positive integers.
The set $\PS(\alpha)$ has infinitely many solutions of the equation $x+y=z$ if $\alpha=1,2$, 
and has no solutions of the same equation if $\alpha\ge3$ is an integer (Fermat's last theorem) \cite{Wiles}.
When $\alpha=2$, such a solution $(x,y,z)$ is called a \textit{Pythagorean triple}, 
and several asymptotic formulas are known for the number of Pythagorean triples.
For example, the number of Pythagorean triples with hypotenuse less than $x$ 
is estimated as \cite{Stronina} 
\begin{align*}
	&\quad \#\{ (l,m,n)\in\mathbb{N}^3 : n<x,\ l^2+m^2=n^2 \}\\
	&= \frac{1}{\pi}x\log x + Bx + O\bigl( x^{1/2}\exp(-C(\log x)^{3/5}(\log\log x)^{-1/5}) \bigr)
	\quad (x\to\infty)
\end{align*}
for an explicit constant $B$ and some $C>0$.
If the Riemann hypothesis is true, 
the above error term is improved \cite{NR}.
Also, the number of Pythagorean triples with both legs less than $x$ 
is estimated as \cite{BV} 
\begin{align*}
	&\quad \#\{ (l,m,n)\in\mathbb{N}^3 : l,m<x,\ l^2+m^2=n^2 \}\\
	&= \frac{4\log(1+\sqrt{2})}{\pi^2}x\log x + O(x) \quad (x\to\infty).
\end{align*}
For other asymptotic formulas, 
see \cite{LM, Liu} (primitive Pythagorean triples with perimeter less than $x$) 
and \cite{LM,Wild,DS,Zhai} (primitive Pythagorean triples with area less than $x$).

Consider the case of non-integral $\alpha>1$.
As a special case of \cite[Proposition~5.1]{FW}, 
it is known that, for all $\alpha\in(1,2)$ and sufficiently large $d\in\mathbb{N}$, 
the set $\PS(\alpha)$ has a solution of the equation $d=z-y$.
This fact immediately implies that the set $\PS(\alpha)$ has infinitely many solutions of the equation $x+y=z$ if $\alpha\in(1,2)$.
However, we almost never know an asymptotic formula for the number of such solutions in the case $\alpha\in(1,2)$.
Indeed, we only know the following asymptotic formula \cite[Corollary~1.3]{Yoshida}: 
for every $\alpha\in(1,6/5)$, 
\begin{equation*}
	\lim_{x\to\infty} \frac{\#\{ (l,m,n)\in\mathbb{N}^3 : n<x,\ 
	\lfloor{l^\alpha}\rfloor+\lfloor{m^\alpha}\rfloor=\lfloor{n^\alpha}\rfloor \}}{x^{3-\alpha}}
	= \frac{\Gamma(1+1/\alpha)^2}{(3-\alpha)\Gamma(2/\alpha)},
\end{equation*}
where $\Gamma$ denotes the gamma function.

In this paper, we estimate the number of solutions of the equation $d=z-y$ in $\PS(\alpha)$.
For a real number $\alpha>1$ and an integer $d\ge1$, 
define the number $\cN_\alpha(d)$ as 
\[
\cN_\alpha(d)
= \#\{ (m,n)\in\mathbb{N}^2 : \lfloor{n^\alpha}\rfloor - \lfloor{m^\alpha}\rfloor = d \}.
\]
Note that $\cN_\alpha(d)$ is finite for all $\alpha>1$ and $d\in\mathbb{N}$. 

\begin{theorem}\label{main0}
	Let $1<\alpha<(\sqrt{21}+4)/5$ and $\beta=1/(\alpha-1)$.
	Then 
	\[
	\lim_{d\to\infty} \frac{\cN_\alpha(d)}{d^{\beta-1}}
	= \beta\alpha^{-\beta}\zeta(\beta),
	\]
	where $\zeta$ denotes the Riemann zeta function.
\end{theorem}

Theorem~\ref{main0} yields the following asymptotic result immediately.

\begin{corollary}\label{main1}
	Let $1<\alpha<(\sqrt{21}+4)/5$ and $\beta=1/(\alpha-1)$.
	Then 
	\begin{equation*}
		\lim_{x\to\infty} \frac{\#\{ (l,m,n)\in\mathbb{N}^3 : l<x,\ 
		\lfloor{l^\alpha}\rfloor + \lfloor{m^\alpha}\rfloor = \lfloor{n^\alpha}\rfloor \}}
		{x^{\alpha(\beta-1)+1}}
		= \frac{\beta\alpha^{-\beta}\zeta(\beta)}{\alpha(\beta-1)+1}.
	\end{equation*}
\end{corollary}
\begin{proof}[Proof of Corollary~$\ref{main1}$ assuming Theorem~$\ref{main0}$]
	The above number of triplets $(l,m,n)$ is expressed as 
	$\sum_{1\le l<x} \cN_\alpha(\lfloor{l^\alpha}\rfloor)$.
	Hence, Corollary~\ref{main1} follows form Theorem~\ref{main0}.
\end{proof}

The set $\PS(\alpha)$ is called a \textit{Piatetski-Shapiro sequence} when $\alpha>1$ is non-integral.
When $f$ is a positive-integer-valued function with suitable properties, 
we can apply Theorem~\ref{main0} to the equation $f(x)+y=z$ 
in the same way as the proof of Corollary~\ref{main1}.
Although Corollary~\ref{main1} does not cover the case $\alpha\ge(\sqrt{21}+4)/5$, 
we expect that the asymptotic formula in Corollary~\ref{main1} should be true for every $1<\alpha\le2$.
If formally substituting $\alpha=2$ in Corollary~\ref{main1}, 
then we obtain 
\[
\lim_{x\to\infty} \frac{\#\{ (l,m,n)\in\mathbb{N}^3 : l<x,\ l^2 + m^2 = n^2 \}}{x}
= \infty.
\]
Actually, one can verify the asymptotic formula 
\begin{align*}
	&\quad \#\{ (l,m,n)\in\mathbb{N}^3 : l<x,\ l^2+m^2=n^2 \}\\
	&= \frac{x(\log x)^2}{\pi^2} + O(x\log x) \quad (x\to\infty)
\end{align*}
in a similar way to \cite{BV}.

\subsection{Additive energies of Piatetski-Shapiro sequences}

Recently, many researchers have paid attention to the \textit{additive energies} of strictly increasing integer sequences $(a_n)_{n=1}^N$, i.e., 
the number of quadruples $(n_1,n_2,n_3,n_4)\in\mathbb{N}^4$ such that $a_{n_1}+a_{n_2}=a_{n_3}+a_{n_4}$ and $n_1,n_2,n_3,n_4\le N$, 
since it is closely related to the metric Poissonian property \cite{Aistleitner1, Aistleitner2, Bloom, Larcher}.
From now on, we denote by $\cE_\alpha(N)$ the additive energy of $(\lfloor{n^\alpha}\rfloor)_{n=1}^N$.
Several bounds for $\cE_\alpha(N)$ are known in existing studies.
For example, it follows from \cite[p.~505]{Aistleitner1} (essentially \cite[Theorem~2]{Robert-Sargos}) that 
$\cE_\alpha(N) \ll_\epsilon N^{2+\epsilon} + N^{4-\alpha+\epsilon}$ for all $\alpha>1$, $\epsilon>0$ and $N\in\mathbb{N}$.
Also, it is known that $\cE_\alpha(N) \ll_\alpha N^{4-\alpha}\log(N+1)$ for all $\alpha\in(1,3/2]$ and $N\in\mathbb{N}$ \cite[p.~1000]{Garaev-Kueh}.
Moreover, it is easy to check that $\cE_\alpha(N) \gg N^2+N^{4-\alpha}$ for all $\alpha\ge1$ and $N\in\mathbb{N}$ (for instance, see \cite[p.~1000]{Garaev-Kueh}).
However, this lower bound is not equal to the above upper bounds as the growth rate.


In this paper, we estimate $\cE_\alpha(N)$ and obtain an upper bound for $\cE_\alpha(N)$ when $1<\alpha\le4/3$.

\begin{theorem}\label{main2}
	For all $\alpha\in(1,4/3]$ and $N\in\mathbb{N}$, 
	the additive energy $\cE_\alpha(N)$ of $(\lfloor{n^\alpha}\rfloor)_{n=1}^N$ is equal to $O_\alpha(N^{4-\alpha})$.
\end{theorem}

The upper bound in Theorem~\ref{main2} is best possible, since the known lower bound $N^2+N^{4-\alpha}$ is greater than $N^{4-\alpha}$. 
We prove Theorem~\ref{main2} by estimating a variant of $\cN_\alpha(d)$ in Section~\ref{AE}.

However, one can also estimate $\cE_\alpha(N)$ as follows.
Let $\cR_\alpha(N)$ be the number of solutions of the equation $x+y=N$ in $\PS(\alpha)$.
Then 
\begin{equation*}
	\cE_\alpha(N) \le \sum_{2\le n\le2N^\alpha} \cR_\alpha(n)^2.
\end{equation*}
Assume $\alpha\in(1,4/3)$.
Since $\cR_\alpha(N) \asymp N^{2/\alpha-1}$ as $N\to\infty$ \cite[Corollary~1.6]{Yoshida}, 
it can easily be checked that $\cE_\alpha(N) \ll_\alpha N^{4-\alpha}$.
This way is different from the proof in Section~\ref{AE}, and 
is easier to come up with.

\subsection{Related work on the number of solutions of a linear equation in $\PS(\alpha)$}

There are existing studies that examined whether a linear equation in two or three variables has infinitely many solutions or not.
An equation $f(x_1,\ldots,x_k)=0$ is called \textit{solvable} in a subset $\cS$ of $\mathbb{N}$ 
if $\cS^k$ contains infinitely many pairwise distinct tuples $(x_1,\ldots,x_k)$ with $f(x_1,\ldots,x_k)=0$.
Matsusaka and Saito \cite{MS} showed that, for all $t>s>2$ and $a,b,c\in\mathbb{N}$, 
the set of all $\alpha\in[s,t]$ such that the equation $ax+by=cz$ is solvable in $\PS(\alpha)$, 
has positive Hausdorff dimension.
Glasscock \cite{Glasscock2} showed that 
if the equation $y=ax+b$ with real numbers $a\not=0,1$ and $b$ is solvable in $\mathbb{N}$, 
then, for Lebesgue-a.e.\ $\alpha\in(1,2)$ (resp.\ $\alpha>2$), 
the equation $y=ax+b$ is solvable (resp.\ not solvable) in $\mathrm{PS}(\alpha)$.
Saito \cite{Saito1} improved Glasscock's result in the case $a>b\ge0$ as follows.
Let $a\not=1$ and $b$ be real numbers with $a>b\ge0$.
If the equation $y=ax+b$ is solvable in $\mathbb{N}$, 
then (i) for all $\alpha\in(1,2)$, the equation $y=ax+b$ is solvable in $\PS(\alpha)$; 
(ii) for all $t>s>2$, the set of all $\alpha\in[s,t]$ such that the equation $y=ax+b$ is solvable in $\PS(\alpha)$, 
has the Hausdorff dimension $2/s$.
Moreover, Saito \cite{Saito2} investigated the Hausdorff dimension of the set of $\alpha\in[s,t]$ such that 
a linear equation $y=a_1x_1+\cdots+a_nx_n$ with positive coefficients $a_1,\ldots,a_n$ has infinitely many solutions in $\PS(\alpha)$.
For example, he showed the following statement \cite[Theorem~1.2]{Saito2}: 
for Lebesgue-a.e.\ $\alpha>3$, the equation $x+y=z$ has at most finitely many solutions in $\PS(\alpha)$. 

\section{More general statements}
\subsection{Arithmetic progressions with fixed common difference}

We extend the definition of $\cN_\alpha(d)$ a little.
For a real number $\alpha>1$ and integers $d\ge1$ and $k\ge2$, 
define the number $\cN_{\alpha,k}(d)$ as 
\[
\cN_{\alpha,k}(d) = \#\Biggl\{ (n,r)\in\mathbb{N}^2 : 
\begin{array}{l}
	\forall j=0,1,\ldots,k-1,\\
	\lfloor{(n+rj)^\alpha}\rfloor = \lfloor{n^\alpha}\rfloor+dj
\end{array}
\Biggr\}.
\]
In other words, $\cN_{\alpha,k}(d)$ is the number of arithmetic progressions $P$ of length $k$ ($k$-APs) such that 
$(\lfloor{n^\alpha}\rfloor)_{n\in P}$ is an AP with common difference $d$.
The number $\cN_\alpha(d)$ is equal to $\cN_{\alpha,2}(d)$.
Note that $\cN_{\alpha,k}(d)$ is finite for all $\alpha>1$, all integers $d\ge1$ and $k\ge2$.
The following theorem is an extension of Theorem~\ref{main0}.

\begin{theorem}\label{main0'}
	Let $1<\alpha<(\sqrt{21}+4)/5$ and $\beta=1/(\alpha-1)$.
	Then, for every integer $k\ge2$, 
	\[
	\lim_{d\to\infty} \frac{\cN_{\alpha,k}(d)}{d^{\beta-1}}
	= \frac{\beta\alpha^{-\beta}\zeta(\beta)}{k-1}.
	\]
\end{theorem}

\textbf{Notation.} From now on, we denote by $\{x\}=x-\lfloor{x}\rfloor$ the fractional part of a real number $x$, 
and use the notations ``$o(\cdot)$, $O(\cdot)$, $\sim$, $\asymp$, $\ll$'' in the usual sense.
If implicit constants depend on parameters $a_1,\ldots,a_n$, 
we often write ``$O_{a_1,\ldots,a_n}(\cdot)$, $\asymp_{a_1,\ldots,a_n}$, $\ll_{a_1,\ldots,a_n}$'' 
instead of ``$O(\cdot)$, $\asymp$, $\ll$''.
Also, the expression $a/bc$ always means $\frac{a}{bc}$ and does not mean $\frac{a}{b}c$.

\vskip1ex
Theorem~\ref{main0'} is derived from the following propositions and lemmas immediately.

\begin{proposition}\label{liminf}
	Let $1<\alpha<2$ and $\beta=1/(\alpha-1)$.
	Then, for every integer $k\ge2$, 
	\[
	\liminf_{d\to\infty} \frac{\cN_{\alpha,k}(d)}{d^{\beta-1}}
	\ge \frac{\beta\alpha^{-\beta}\zeta(\beta)}{k-1}.
	\]
\end{proposition}

\begin{proposition}\label{limsup}
	Let $1<\alpha<2$ and $\beta=1/(\alpha-1)$.
	For an integer $r\ge1$, 
	define the strictly increasing function $f_r\colon [0,\infty)\to[r^\alpha,\infty)$ as $f_r(x)=(x+r)^\alpha-x^\alpha$.
	Then, for every integer $k\ge2$, 
	\[
	\limsup_{d\to\infty} \frac{\cN_{\alpha,k}(d)}{d^{\beta-1}}
	\le \frac{\beta\alpha^{-\beta}\zeta(\beta)}{k-1}
	+ \limsup_{d\to\infty} \frac{E_1(d)+E_2(d)}{d^{\beta-1}},
	\]
	where $E_1(d)$ and $E_2(d)$ are defined as 
	\[
	E_1(d) = \#\bigl\{ r\in\mathbb{N} : r\le d^{1/\alpha}/4,\ \{f_r^{-1}(d)\} + C_1d^{\beta-1}/r^\beta > 1 \bigr\}
	\]
	and 
	\[
	E_2(d) = \#\bigl\{ n\in\mathbb{N} : n<C_2d^{1/\alpha},\ \{(n^\alpha+d)^{1/\alpha}\} + 2d^{1/\alpha-1} > 1 \bigr\}
	\]
	with suitable constants $C_1,C_2>0$ only depending on $\alpha$.
\end{proposition}

\begin{lemma}\label{lemU4}
	Let $1<\alpha<1+1/\sqrt{2}$, $\beta=1/(\alpha-1)$ and $C_1>0$.
	Then $E_1(d)$ defined in Proposition~$\ref{limsup}$ is equal to $o(d^{\beta-1})$ as $d\to\infty$.
\end{lemma}

\begin{lemma}\label{lemU4'}
	Let $1+1/\sqrt{2}\le\alpha<(\sqrt{21}+4)/5$, $\beta=1/(\alpha-1)$ and $C_1>0$.
	Then $E_1(d)$ defined in Proposition~$\ref{limsup}$ is equal to $o(d^{\beta-1})$ as $d\to\infty$.
\end{lemma}

\begin{lemma}\label{lemU5}
	Let $1<\alpha<(\sqrt{10}+2)/3$, $\beta=1/(\alpha-1)$ and $C_2>0$.
	Then $E_2(d)$ defined in Proposition~$\ref{limsup}$ is equal to $o(d^{\beta-1})$ as $d\to\infty$.
\end{lemma}

The above propositions and lemmas are proved in Sections~\ref{proof1} and \ref{proof2}.

\subsection{Statements in order to prove Theorem~$\ref{main2}$}

To prove Theorem~\ref{main2}, we extend the definition of $\cN_\alpha(d)$ a little in another way.
For a real number $\alpha>1$ and integers $d,l\ge1$, 
define the number $\cN_\alpha^{\ge l}(d)$ as 
\[
\cN_\alpha^{\ge l}(d)
= \#\{ (n,r)\in\mathbb{N}^2 : r\ge l,\ \lfloor{(n+r)^\alpha}\rfloor - \lfloor{n^\alpha}\rfloor = d \}.
\]
If $l=1$, then $\cN_\alpha^{\ge l}(d)=\cN_{\alpha,2}(d)=\cN_\alpha(d)$.
Theorem~\ref{main2}, which is proved in Section~\ref{AE}, is derived from the following proposition and lemmas.

\begin{proposition}\label{UB}
	Let $1<\alpha<2$ and $\beta=1/(\alpha-1)$.
	Then, for all integers $d,l\ge1$, 
	\[
	\cN_\alpha^{\ge l}(d) \ll_\alpha d^{\beta-1}l^{1-\beta} + d^{2/\alpha-1} + E_1^{\ge l}(d-1) + E_2(d-1),
	\]
	where $E_2(d)$ is defined in Proposition~$\ref{limsup}$, and $E_1^{\ge l}(d)$ is defined as 
	\[
	E_1^{\ge l}(d) = \#\bigl\{ r\in\mathbb{N} : l\le r\le d^{1/\alpha}/4,\ \{f_r^{-1}(d)\} + C_1d^{\beta-1}/r^\beta > 1 \bigr\}
	\]
	with the same function $f_r$ and the same constant $C_1$ as in Proposition~$\ref{limsup}$.
\end{proposition}

\begin{lemma}\label{lemU6}
	Let $1<\alpha\le3/2$ and $\beta=1/(\alpha-1)$.
	Then, for all integers $d,l\ge1$, 
	\[
	E_1^{\ge l}(d) \ll_\alpha d^{\beta-1}l^{1-\beta} + d^{\beta/3}l^{(1-\beta)/3} + d^{2/\alpha-1},
	\]
	where $E_1^{\ge l}(d)$ is defined in Proposition~$\ref{UB}$.
\end{lemma}

\begin{lemma}\label{lemU7}
	Let $1<\alpha\le4/3$ and $\beta=1/(\alpha-1)$.
	Then, for all integers $d\ge1$, 
	\[
	E_2(d) \ll_\alpha d^{2/\alpha-1},
	\]
	where $E_2(d)$ is defined in Proposition~$\ref{limsup}$.
\end{lemma}

The above proposition and lemmas are proved in Section~\ref{AE}.

\section{Equidistribution modulo $1$ and exponential sums}

The following lemma is a key point in proving Propositions~\ref{liminf} and \ref{limsup}.

\begin{lemma}[Equidistribution modulo $1$]\label{lemWeyl}
	Let $1<\alpha<2$, $\beta=1/(\alpha-1)$, $r\in\mathbb{N}$, 
	$c_1<c_2$ and $c_2-c_1\in\mathbb{N}$.
	Then, for every convex set $\cC$ of $[0,1)^2$, 
	\begin{align*}
		\lim_{d\to\infty} \frac{1}{d^{\beta-1}}
		\#\Biggl\{ n\in\mathbb{N} : 
		\begin{array}{c}
			\displaystyle \Bigl( \frac{d+c_1}{r\alpha} \Bigr)^\beta \le n < \Bigl( \frac{d+c_2}{r\alpha} \Bigr)^\beta,\vspace{1ex}\\
			(\{n^\alpha\},\{r\alpha n^{\alpha-1}\})\in\cC
		\end{array}
		\Biggr\}\\
		= \frac{\beta(c_2-c_1)}{(r\alpha)^\beta}\mu(\cC),
	\end{align*}
	where $\mu$ denotes the Lebesgue measure on $\mathbb{R}^2$.
\end{lemma}

Recently, Saito and Yoshida \cite[Lemma~5.7]{SY2} investigated 
the distribution of the sequence $\bigl( (n^\alpha,r\alpha n^{\alpha-1}) \bigr)_{n=1}^\infty$ modulo $1$ on short intervals.
However, we cannot use their lemma here because the length 
\[
\Bigl( \frac{d+c_2}{r\alpha} \Bigr)^\beta - \Bigl( \frac{d+c_1}{r\alpha} \Bigr)^\beta
\sim \frac{\beta(c_2-c_1)}{(r\alpha)^\beta}d^{\beta-1}
\]
of the interval in Lemma~\ref{lemWeyl} is too short.
Nevertheless, the above equidistribution holds due to the assumption $c_2-c_1\in\mathbb{N}$.
We prove Lemma~\ref{lemWeyl} at the end of this section.

To prove Lemma~\ref{lemWeyl}, we need to estimate exponential sums.
Denote by $e(x)$ the function $e^{2\pi ix}$, 
and by $|\cI|$ the length of an interval $\cI$ of $\mathbb{R}$.
The following lemmas are useful to estimate exponential sums.

\begin{lemma}[van der Corput]\label{2ndderiv}
	Let $\cI$ be an interval of $\mathbb{R}$ with $|\cI|\ge1$ and 
	$f\colon \cI\to\mathbb{R}$ be a $C^2$ function, 
	and let $c\ge1$.
	If $\lambda_2>0$ satisfies that 
	\[
	\lambda_2 \le |f''(x)| \le c\lambda_2
	\] 
	for all $x\in\cI$, then 
	\[
	\sum_{n\in\cI\cap\mathbb{Z}} e(f(n))
	\ll_c \abs{\cI}\lambda_2^{1/2} + \lambda_2^{-1/2}.
	\]
\end{lemma}

\begin{lemma}[Sargos--Gritsenko]\label{3rdderiv}
	Let $\cI$ be an interval of $\mathbb{R}$ with $|\cI|\ge1$ and 
	$f\colon \cI\to\mathbb{R}$ be a $C^3$ function, 
	and let $c\ge1$.
	If $\lambda_3>0$ satisfies that 
	\[
	\lambda_3 \le |f'''(x)| \le c\lambda_3
	\]
	for all $x\in\cI$, then 
	\[
	\sum_{n\in\cI\cap\mathbb{Z}} e(f(n))
	\ll_c \abs{\cI}\lambda_3^{1/6} + \lambda_3^{-1/3}.
	\]
\end{lemma}

Lemma~\ref{2ndderiv} is called the second derivative test; 
for its proof, see \cite[Theorem~2.2]{GK} for instance.
Lemma~\ref{3rdderiv} was shown by Sargos \cite{Sargos} and Gritsenko \cite{Gritsenko} independently.
Now, let us prove Lemma~\ref{lemWeyl} by using Lemma~\ref{2ndderiv}.

\begin{proof}[Proof of Lemma~$\ref{lemWeyl}$]
	Define the function $f$ as $f(x)=x^\alpha$.
	If the following criterion holds, 
	Lemma~\ref{lemWeyl} follows in the same way as Weyl's equidistribution theorem.
	\textit{Weyl's criterion}: for every non-zero $(h_1,h_2)\in\mathbb{Z}^2$, 
	\begin{equation}
		\lim_{d\to\infty} \frac{1}{N-M}\sum_{n=M}^{N-1} e(h_1 f(n) + h_2 rf'(n))
		= 0, \label{eqWeyl}
	\end{equation}
	where 
	\[
	M=M(d) \coloneqq \Bigl\lceil{\Bigl( \frac{d+c_1}{r\alpha} \Bigr)^\beta}\Bigr\rceil \quad\text{and}\quad
	N=N(d) \coloneqq \Bigl\lceil{\Bigl( \frac{d+c_2}{r\alpha} \Bigr)^\beta}\Bigr\rceil.
	\]
	First, we show \eqref{eqWeyl} when $h_1=0$ and $h_2\not=0$.
	Let $h_1=0$, and let $h_2\not=0$ be an integer.
	Then 
	\begin{align}
		&\quad \abs{\frac{1}{N-M}\sum_{n=M}^{N-1} e(h_2 rf'(n))}\nonumber\\
		\begin{split}
			&\le \frac{1}{N-M}\abs{\sum_{n=M}^{N-1} e(h_2 rf'(n))
			- \sum_{n=M}^{N-1} e\bigl( h_2 r(f'(M) + (n-M)f''(M)) \bigr)}\\
			&\quad+ \frac{1}{N-M}\abs{\sum_{n=M}^{N-1} e\bigl( h_2 r(f'(M) + (n-M)f''(M)) \bigr)}.
		\end{split}\label{eqW1}
	\end{align}
	Taylor's theorem implies that for every integer $M<n<N$, 
	there exists $\theta_n\in(M,n)$ such that 
	\begin{equation*}
		f'(n) = f'(M) + (n-M)f''(M) + \frac{(n-M)^2}{2}f'''(\theta_n).
	\end{equation*}
	Due to this equality and the inequality $\abs{e(y)-e(x)} \le 2\pi\abs{y-x}$, 
	the first term of \eqref{eqW1} bounded from above by 
	\begin{align}
		&\quad \frac{2\pi\abs{h_2}r}{N-M}\sum_{n=M}^{N-1} \abs{f'(n) - (f'(M) + (n-M)f''(M))}\nonumber\\
		&\le \frac{2\pi\abs{h_2}r}{N-M}\sum_{n=M}^{N-1} \frac{(n-M)^2}{2}\abs{f'''(\theta_n)}. \label{eqW2}
	\end{align}
	To estimate \eqref{eqW2}, we note that 
	\[
	N-M \sim \Bigl( \frac{d+c_2}{r\alpha} \Bigr)^\beta - \Bigl( \frac{d+c_1}{r\alpha} \Bigr)^\beta
	\sim \frac{\beta(c_2-c_1)}{r\alpha}\Bigl( \frac{d}{r\alpha} \Bigr)^{\beta-1} \quad (d\to\infty).
	\]
	Eq.~\eqref{eqW2} is bounded from above by 
	\[
	\pi\abs{h_2}r(N-M)^2\abs{f'''(M)}
	\ll_{\alpha,c_1,c_2,h_2} d^{2(\beta-1)}d^{\beta(\alpha-3)} = d^{-1},
	\]
	whence \eqref{eqW2} is equal to $O_{\alpha,c_1,c_2,h_2}(1/d)$, and so is the first term of \eqref{eqW1}.
	Also, the second term of \eqref{eqW1} is equal to 
	\begin{equation}
	\begin{split}
		&\quad \frac{1}{N-M}\abs{\frac{1-e\bigl( h_2 r(N-M)f''(M) \bigr)}{1-e(h_2 rf''(M))}}\\
		&= \frac{1}{N-M}\abs{\frac{\sin\bigl( \pi h_2 r(N-M)f''(M) \bigr)}{\sin(\pi h_2 rf''(M))}}.
	\end{split}\label{eqW3}
	\end{equation}
	Since the relations $N-M \gg_{\alpha,r,c_1} d^{\beta-1}$ and $f''(M)\in(0,\pi/2)$ hold for sufficiently large $d\ge1$, 
	it follows that 
	\begin{align*}
		(N-M)\abs{\sin(\pi h_2 rf''(M))}
		&\ge (N-M)\cdot\frac{2}{\pi}\cdot\pi \abs{h_2}rf''(M)\\
		&\gg_{\alpha,r,c_1} d^{\beta-1}d^{\beta(\alpha-2)} = 1
	\end{align*}
	for sufficiently large $d\ge1$.
	Thus, \eqref{eqW3} is equal to 
	\[
	O_{\alpha,r,c_1}\Bigl( \abs{\sin\bigl( \pi h_2 r(N-M)f''(M) \bigr)} \Bigr)
	\]
	for sufficiently large $d\ge1$, 
	which vanishes as $d\to\infty$ because 
	\begin{align*}
		r(N-M)f''(M)
		&\sim r \cdot \frac{\beta(c_2-c_1)}{r\alpha}\Bigl( \frac{d}{r\alpha} \Bigr)^{\beta-1}
		\cdot \alpha(\alpha-1)\Bigl( \frac{d}{r\alpha} \Bigr)^{\beta(\alpha-2)}\\
		&= (c_2-c_1)\Bigl( \frac{d}{r\alpha} \Bigr)^{\beta-1}\Bigl( \frac{d}{r\alpha} \Bigr)^{1-\beta}\\
		&= c_2-c_1\in\mathbb{N} \quad (d\to\infty).
	\end{align*}
	Therefore, \eqref{eqW3} vanishes as $d\to\infty$, and so does the second term of \eqref{eqW1}.
	\par
	Next, we show \eqref{eqWeyl} when $h_1\not=0$.
	Let $h_1\not=0$ and $h_2$ be integers.
	The second derivative of the function $g(x) \coloneqq h_1 f(x) + h_2 rf'(x)$ satisfies that 
	\begin{align*}
		\abs{g''(x)} &\le \abs{h_1}\alpha(\alpha-1)x^{\alpha-2}(1 + \abs{h_2/h_1}r(2-\alpha)x^{-1})\\
		&\le 2\abs{h_1}\alpha(\alpha-1)x^{\alpha-2}
	\end{align*}
	and 
	\begin{align*}
		\abs{g''(x)} &\ge \abs{h_1}\alpha(\alpha-1)x^{\alpha-2}(1 - \abs{h_2/h_1}r(2-\alpha)x^{-1})\\
		&\ge (1/2)\abs{h_1}\alpha(\alpha-1)x^{\alpha-2}
	\end{align*}
	for every $x\ge M$ with sufficiently large $d\ge1$.
	Thus, for every $M\le x\le N$ with sufficiently large $d\ge1$, 
	\[
	\abs{g''(x)} \asymp_{\alpha,r,h_1} d^{\beta(\alpha-2)} = d^{1-\beta}.
	\]
	By Lemma~\ref{2ndderiv} and the inequality $N-M \gg_{\alpha,r,c_1} d^{\beta-1}$, 
	we obtain that 
	\begin{align*}
		&\quad \abs{\frac{1}{N-M}\sum_{n=M}^{N-1} e(h_1 f(n) + h_2 rf'(n))}\\
		&\ll_{\alpha,r,h_1} \frac{(N-M)d^{(1-\beta)/2} + d^{(\beta-1)/2}}{N-M}
		\ll_{\alpha,r,c_1} d^{(1-\beta)/2}
	\end{align*}
	for sufficiently large $d\ge1$.
	Therefore, \eqref{eqWeyl} follows.
\end{proof}

\section{Proofs of Propositions~$\ref{liminf}$ and $\ref{limsup}$}\label{proof1}

We prove Propositions~\ref{liminf} and \ref{limsup} by using Lemma~\ref{lemWeyl}.

\begin{proof}[Proof of Proposition~$\ref{liminf}$]
	Let $k\ge2$ be an integer.
	Take arbitrary $\epsilon\in(0,1)$ and $R\in\mathbb{N}$.
	We also take $x_0=x_0(k,\epsilon,R)>0$ such that 
	\begin{equation}
		\frac{R^2(k-1)^2}{2}\alpha(\alpha-1)x_0^{\alpha-2} \le \epsilon. \label{eqL1}
	\end{equation}
	Let us show that 
	\[
	\liminf_{d\to\infty} \frac{\cN_{\alpha,k}(d)}{d^{\beta-1}}
	\ge \frac{(1-\epsilon)^2\beta\alpha^{-\beta}}{k-1}\sum_{r=1}^R \frac{1}{r^\beta}.
	\]
	Define the convex set $\cC_k^{-}(\epsilon)$ of $\mathbb{R}^2$ as 
	\[
	\cC_k^{-}(\epsilon) = \{ (y_0,y_1)\in\mathbb{R}^2 : 0\le y_0<1-\epsilon,\ 0\le y_0+(k-1)y_1<1-\epsilon \}.
	\]
	Note that if $(y_0,y_1)\in\cC_k^{-}(\epsilon)$, 
	then $0\le y_0+jy_1<1-\epsilon$ for all $j=0,1,\ldots,k-1$.
	Taylor's theorem implies that for all integers $n,r\ge1$, $j\ge0$ and $s$, 
	\begin{equation}
	\begin{split}
		(n+rj)^\alpha &= n^\alpha + rj\alpha n^{\alpha-1} + \frac{(rj)^2}{2}\alpha(\alpha-1)(n+rj\theta)^{\alpha-2}\\
		&= \lfloor{n^\alpha}\rfloor + j(\lfloor{r\alpha n^{\alpha-1}}\rfloor+s) + \delta_s(n,r,j),
	\end{split}\label{eqL2}
	\end{equation}
	where $\theta=\theta(n,r,j)\in(0,1)$ and 
	\[
	\delta_s(n,r,j) \coloneqq \{n^\alpha\} + j(\{r\alpha n^{\alpha-1}\}-s) + \frac{(rj)^2}{2}\alpha(\alpha-1)(n+rj\theta)^{\alpha-2}.
	\]
	Thus, if $s\in\mathbb{Z}$, $n\ge x_0$ and $(\{n^\alpha\}, \{r\alpha n^{\alpha-1}\}-s)\in\cC_k^{-}(\epsilon)$, 
	then $0\le\delta_s(n,r,j)<1$ and 
	$\lfloor{(n+rj)^\alpha}\rfloor = \lfloor{n^\alpha}\rfloor + j(\lfloor{r\alpha n^{\alpha-1}}\rfloor+s)$ 
	for all $r=1,2,\ldots,R$ and $j=0,1,\ldots,k-1$.
	(We have $s\in\{0,1\}$ under the same assumptions, but this fact is not necessarily used here.)
	Also, the following equivalence holds: 
	\[
	\lfloor{r\alpha n^{\alpha-1}}\rfloor + s = d
	\iff \Bigl( \frac{d-s}{r\alpha} \Bigr)^\beta \le n < \Bigl( \frac{d-s+1}{r\alpha} \Bigr)^\beta.
	\]
	From the above facts, it follows that for sufficiently large $d\ge1$, 
	\begin{align*}
		\cN_{\alpha,k}(d) &\ge \sum_{r=1}^R \#\Biggl\{ n\in\mathbb{N} : 
		\begin{array}{l}
			\forall j=0,1,\ldots,k-1,\\
			\lfloor{(n+rj)^\alpha}\rfloor = \lfloor{n^\alpha}\rfloor+dj
		\end{array}
		\Biggr\}\\
		&\ge \sum_{r=1}^R \sum_{s\in\mathbb{Z}} \#\Biggl\{ n\ge x_0 : 
		\begin{array}{c}
			\displaystyle \Bigl( \frac{d-s}{r\alpha} \Bigr)^\beta \le n < \Bigl( \frac{d-s+1}{r\alpha} \Bigr)^\beta,\vspace{1ex}\\
			(\{n^\alpha\}, \{r\alpha n^{\alpha-1}\}-s)\in\cC_k^{-}(\epsilon)
		\end{array}
		\Biggr\},
	\end{align*}
	where the sum $\sum_{s\in\mathbb{Z}}$ is a finite sum due to the boundedness of $\cC_k^{-}(\epsilon)$ 
	(or due to $s\in\{0,1\}$ as already stated).
	Using Lemma~\ref{lemWeyl}, we obtain 
	\begin{align*}
		&\quad \liminf_{d\to\infty} \frac{\cN_{\alpha,k}(d)}{d^{\beta-1}}
		\ge \sum_{r=1}^R \sum_{s\in\mathbb{Z}} \frac{\beta}{(r\alpha)^\beta}\mu\Bigl( \cC_k^{-}(\epsilon)\cap\bigl( [0,1)\times[-s,1-s) \bigr) \Bigr)\\
		&= \sum_{r=1}^R \frac{\beta}{(r\alpha)^\beta}\mu\bigl( \cC_k^{-}(\epsilon) \bigr)
		= \sum_{r=1}^R \frac{\beta}{(r\alpha)^\beta}\cdot\frac{(1-\epsilon)^2}{k-1}
		= \frac{(1-\epsilon)^2\beta\alpha^{-\beta}}{k-1}\sum_{r=1}^R \frac{1}{r^\beta}.
	\end{align*}
	Finally, letting $\epsilon\to+0$ and $R\to\infty$, 
	we complete the proof.
\end{proof}

\begin{proof}[Proof of Proposition~$\ref{limsup}$]
	Let $k\ge2$ be an integer.
	Take arbitrary $\epsilon\in(0,1)$ and $R\in\mathbb{N}$.
	We also take $x_0=x_0(k,\epsilon,R)>0$ that satisfies \eqref{eqL1}.
	Let us show that 
	\begin{equation}
	\begin{split}
		\limsup_{d\to\infty} \frac{\cN_{\alpha,k}(d)}{d^{\beta-1}}
		&\le \frac{(1+\epsilon)\beta\alpha^{-\beta}}{k-1}\sum_{r=1}^R \frac{1}{r^\beta}\\
		&\quad+ \limsup_{d\to\infty} \frac{E_1(d)+E_2(d)}{d^{\beta-1}} + O_\alpha\biggl( \sum_{r>R} \frac{1}{r^\beta} \biggr),
	\end{split}\label{eqU01}
	\end{equation}
	where the implicit constant only depends on $\alpha\in(1,2)$.
	Define the convex set $\cC_k^{+}(\epsilon)$ of $\mathbb{R}^2$ as 
	\[
	\cC_k^{+}(\epsilon) = \{ (y_0,y_1)\in\mathbb{R}^2 : 0\le y_0<1,\ -\epsilon\le y_0+(k-1)y_1<1 \}.
	\]
	\par\setcounter{count}{0}
	\textbf{Step~\num.}
	Take $x_0=x_0(k,\epsilon,R)>0$ that satisfies \eqref{eqL1}.
	We show that if integers $d\ge1$, $n\ge x_0$ and $1\le r\le R$ satisfy 
	$\lfloor{(n+rj)^\alpha}\rfloor = \lfloor{n^\alpha}\rfloor+dj$ for all $j=0,1,\ldots,k-1$, 
	then there exists an integer $s<d$ such that 
	\begin{itemize}
		\item
		the point $(\{n^\alpha\}, \{r\alpha n^{\alpha-1}\}-s)$ lies in $\cC_k^{+}(\epsilon)$, and 
		\item
		$\displaystyle \Bigl( \frac{d-s}{r\alpha} \Bigr)^\beta \le n < \Bigl( \frac{d-s+1}{r\alpha} \Bigr)^\beta$.
	\end{itemize}
	Set $s=d-\lfloor{r\alpha n^{\alpha-1}}\rfloor$.
	Then the second condition holds.
	Since Taylor's theorem implies \eqref{eqL2}, 
	it follows that 
	\begin{align*}
		(n+rj)^\alpha &= \lfloor{n^\alpha}\rfloor + j(\lfloor{r\alpha n^{\alpha-1}}\rfloor+s) + \delta_s(n,r,j)\\
		&= \lfloor{n^\alpha}\rfloor + dj + \delta_s(n,r,j)
		= \lfloor{(n+rj)^\alpha}\rfloor + \delta_s(n,r,j)
	\end{align*}
	for all $j=0,1,\ldots,k-1$.
	This yields that $0\le\delta_s(n,r,j)<1$ for all $j=0,1,\ldots,k-1$.
	Thus, the point $(\{n^\alpha\}, \{r\alpha n^{\alpha-1}\}-s)$ lies in $\cC_k^{+}(\epsilon)$.
	\par
	\textbf{Step~\num.}
	For an integer $d\ge1$, define the number $E_0(d,R)$ as 
	\begin{equation}
		E_0(d,R) = \#\{ (n,r)\in\mathbb{N}^2 : r>R,\ \lfloor{(n+r)^\alpha}\rfloor - \lfloor{n^\alpha}\rfloor = d \}.
		\label{eqU20}
	\end{equation}
	By Step~1, we have that for sufficiently large $d\ge1$, 
	\begin{align*}
		\cN_{\alpha,k}(d) &= \sum_{r=1}^R \#\Biggl\{ n\in\mathbb{N} : 
		\begin{array}{l}
			\forall j=0,1,\ldots,k-1,\\
			\lfloor{(n+rj)^\alpha}\rfloor = \lfloor{n^\alpha}\rfloor+dj
		\end{array}
		\Biggr\}\\
		&\quad+ E_0(d,R)\\
		&\le \sum_{r=1}^R \sum_{s\in\mathbb{Z}} \#\Biggl\{ n\ge x_0 : 
		\begin{array}{c}
			\displaystyle \Bigl( \frac{d-s}{r\alpha} \Bigr)^\beta \le n < \Bigl( \frac{d-s+1}{r\alpha} \Bigr)^\beta,\vspace{1ex}\\
			(\{n^\alpha\}, \{r\alpha n^{\alpha-1}\}-s)\in\cC_k^{+}(\epsilon)
		\end{array}
		\Biggr\}\\
		&\quad+ Rx_0 + E_0(d,R),
	\end{align*}
	where the sum $\sum_{s\in\mathbb{Z}}$ is a finite sum due to the boundedness of $\cC_k^{+}(\epsilon)$.
	Lemma~\ref{lemWeyl} implies that 
	\begin{align*}
		\limsup_{d\to\infty} \frac{\cN_{\alpha,k}(d)}{d^{\beta-1}}
		&\le \sum_{r=1}^R \sum_{s\in\mathbb{Z}} \frac{\beta}{(r\alpha)^\beta}\mu\Bigl( \cC_k^{+}(\epsilon)\cap\bigl( [0,1)\times[-s,1-s) \bigr) \Bigr)\\
		&\quad+ \limsup_{d\to\infty} \frac{E_0(d,R)}{d^{\beta-1}}\\
		&= \sum_{r=1}^R \frac{\beta}{(r\alpha)^\beta}\mu\bigl( \cC_k^{+}(\epsilon) \bigr)
		+ \limsup_{d\to\infty} \frac{E_0(d,R)}{d^{\beta-1}}\\
		&= \frac{(1+\epsilon)\beta\alpha^{-\beta}}{k-1}\sum_{r=1}^R \frac{1}{r^\beta}
		+ \limsup_{d\to\infty} \frac{E_0(d,R)}{d^{\beta-1}}.
	\end{align*}
	Also, it follows that 
	\begin{align}
		E_0(d,R)
		&\le \#\{ (n,r)\in\mathbb{N}^2 : r>R,\ d-1<f_r(n)<d+1 \}\nonumber\\
		\begin{split}
			&\le \#\{ (n,r)\in\mathbb{N}^2 : R<r\le(d-1)^{1/\alpha}/4,\ d-1 < f_r(n) < d+1 \}\\
			&\quad+ \#\{ (n,r)\in\mathbb{N}^2 : r>(d-1)^{1/\alpha}/4,\ d-1 < f_r(n) < d+1 \}.
		\end{split}\label{eqU12}
	\end{align}
	\par
	\textbf{Step~\num.}
	Let us estimate the first term of \eqref{eqU12}.
	For every $1\le r\le(d-1)^{1/\alpha}$, 
	the inverse function of $f_r\colon [0,\infty)\to[r^\alpha,\infty)$ is defined, 
	and so are the values $f_r^{-1}(d-1)$ and $f_r^{-1}(d+1)$.
	Thus, 
	\begin{align}
		&\quad \#\{ (n,r)\in\mathbb{N}^2 : R<r\le(d-1)^{1/\alpha}/4,\ d-1 < f_r(n) < d+1 \}\nonumber\\
		&= \sum_{R<r\le(d-1)^{1/\alpha}/4}
		\#\{ n\in\mathbb{N} : f_r^{-1}(d-1) < n < f_r^{-1}(d+1) \}\nonumber\\
		&\le \sum_{R<r\le(d-1)^{1/\alpha}/4}
		\bigl( \lfloor{f_r^{-1}(d+1)}\rfloor - \lfloor{f_r^{-1}(d-1)}\rfloor \bigr)\nonumber\\
		\begin{split}
			&= \sum_{R<r\le(d-1)^{1/\alpha}/4} F_r(d-1)\\
			&\quad+ \sum_{R<r\le(d-1)^{1/\alpha}/4} \bigl( \{f_r^{-1}(d-1)\} - \{f_r^{-1}(d+1)\} \bigr),
		\end{split}\label{eqU13}
	\end{align}
	where $F_r(x) \coloneqq f_r^{-1}(x+2)-f_r^{-1}(x)$.
	The mean value theorem implies that 
	\[
	F_r(d-1) = \frac{2}{(f'_r\circ f_r^{-1})(d+\theta)}
	= \frac{2}{f'_r(y_\theta)}
	< \frac{2(y_\theta+r)^{2-\alpha}}{\alpha(\alpha-1)r},
	\]
	where $\theta=\theta(r,d)\in(-1,1)$ and $y_\theta \coloneqq f_r^{-1}(d+\theta)$.
	If $d\ge1$ and $1\le r\le(d-1)^{1/\alpha}/4<d^{1/\alpha}$, 
	then the inequalities $r<(d/r)^\beta$, 
	\begin{gather*}
		y_\theta+r < \Bigl( \frac{d+\theta}{r\alpha} \Bigr)^\beta + (d/r)^\beta
		\ll_\alpha (d/r)^\beta,\\
		F_r(d-1) \ll_\alpha r^{-1}(d/r)^{\beta(2-\alpha)} = d^{\beta-1}/r^\beta
	\end{gather*}
	hold. Thus, the first sum of \eqref{eqU13} is equal to $O_\alpha( d^{\beta-1}\sum_{r>R} 1/r^\beta)$.
	The second sum of \eqref{eqU13} is bounded from above by 
	\begin{align*}
		&\quad \#\bigl\{ r\le(d-1)^{1/\alpha}/4 : \{f_r^{-1}(d-1)\} > \{f_r^{-1}(d+1)\} \bigr\}\\
		&\le \#\bigl\{ r\le(d-1)^{1/\alpha}/4 : \{f_r^{-1}(d-1)\} + F_r(d-1) \ge 1 \bigr\}\\
		&\le \#\bigl\{ r\le(d-1)^{1/\alpha}/4 : \{f_r^{-1}(d-1)\} + C_1(d-1)^{\beta-1}/r^\beta > 1 \bigr\}\\
		&= E_1(d-1),
	\end{align*}
	where $C_1>0$ is a constant only depending on $\alpha$.
	Therefore, the first term of \eqref{eqU12}
	is less than or equals to $O_\alpha(d^{\beta-1}\sum_{r>R} 1/r^\beta) + E_1(d-1)$.
	\par
	\textbf{Step~\num.}
	Let us estimate the second term of \eqref{eqU12}.
	If $d\ge2$, $r>(d-1)^{1/\alpha}/4$ and $f_r(n) < d+1$, 
	then the mean value theorem implies that 
	\begin{align*}
		(d-1)^{1/\alpha}n^{\alpha-1}
		&\ll (d-1)^{1/\alpha}(\alpha/4)n^{\alpha-1}\\
		&< r\alpha n^{\alpha-1} < f_r(n) < d+1 \ll d-1,
	\end{align*}
	whence $n\ll_\alpha (d-1)^{1/\alpha}$.
	Thus, for some constant $C_2>0$ only depending on $\alpha$, 
	\begin{align}
		&\quad \#\{ (n,r)\in\mathbb{N}^2 : r>(d-1)^{1/\alpha}/4,\ d-1 < f_r(n) < d+1 \}\nonumber\\
		&\le \#\Biggl\{ (n,r)\in\mathbb{N}^2 : 
		\begin{array}{c}
			n<C_2(d-1)^{1/\alpha},\\
			(n^\alpha+d-1)^{1/\alpha} < n+r < (n^\alpha+d+1)^{1/\alpha}
		\end{array}
		\Biggr\}\nonumber\\
		&= \sum_{n<C_2(d-1)^{1/\alpha}} \bigl( \lfloor{(n^\alpha+d+1)^{1/\alpha}}\rfloor - \lfloor{(n^\alpha+d-1)^{1/\alpha}}\rfloor \bigr)\nonumber\\
		\begin{split}
			&= \sum_{n<C_2(d-1)^{1/\alpha}} G_n(d-1)\\
			&\quad+ \sum_{n<C_2(d-1)^{1/\alpha}} \bigl( \{(n^\alpha+d-1)^{1/\alpha}\} - \{(n^\alpha+d+1)^{1/\alpha}\} \bigr),
		\end{split}\label{eqU14}
	\end{align}
	where $G_n(x) \coloneqq (n^\alpha+x+2)^{1/\alpha} - (n^\alpha+x)^{1/\alpha}$.
	Since the mean value theorem implies that 
	\[
	G_n(d-1) < (2/\alpha)(n^\alpha+d-1)^{1/\alpha-1} < 2(d-1)^{1/\alpha-1},
	\]
	the first sum of \eqref{eqU14} is bounded from above by 
	\[
	C_2(d-1)^{1/\alpha}\cdot2(d-1)^{1/\alpha-1}
	\ll_\alpha d^{2/\alpha-1}.
	\]
	Moreover, the second sum of \eqref{eqU14} is bounded from above by 
	\begin{align*}
		&\quad \#\bigl\{ n<C_2(d-1)^{1/\alpha} : \{(n^\alpha+d-1)^{1/\alpha}\} > \{(n^\alpha+d+1)^{1/\alpha}\} \bigr\}\\
		&\le \#\bigl\{ n<C_2(d-1)^{1/\alpha} : \{(n^\alpha+d-1)^{1/\alpha}\} + G_n(d-1) \ge 1 \bigr\}\\
		&\le \#\bigl\{ n<C_2(d-1)^{1/\alpha} : \{(n^\alpha+d-1)^{1/\alpha}\} + 2(d-1)^{1/\alpha-1} > 1 \bigr\}\\
		&= E_2(d-1).
	\end{align*}
	Therefore, the second term of \eqref{eqU12} 
	is less than or equal to $O_\alpha(d^{2/\alpha-1}) + E_2(d-1)$.
	\par
	\textbf{Step~\num.}
	By Steps~2--4, the inequality \eqref{eqU01} holds.
	Letting $\epsilon\to+0$ and $R\to\infty$ in \eqref{eqU01}, 
	we complete the proof.
\end{proof}

\section{Discrepancy and preliminary lemmas}

This section is a preparation to prove Lemmas~\ref{lemU4}--\ref{lemU5}.
For a sequence $(x_n)_{n=1}^N$ of real numbers, 
define the \textit{discrepancy} $D(x_1,\ldots,x_N)$ as 
\begin{align*}
	&\quad D(x_1,\ldots,x_N)\\
	&= \sup_{0\le a<b\le1} \abs{\frac{\#\bigl\{ n\in\mathbb{N} : n\le N,\ a\le\{x_n\}<b \bigr\}}{N} - (b-a)}.
\end{align*}
As easily expected, the estimates of $E_1(d)$ and $E_2(d)$ defined in Proposition~\ref{limsup} 
are derived from the estimates of corresponding discrepancies.
To estimate discrepancies, we use the following lemma.

\begin{lemma}[Koksma--Sz\"{u}sz \cite{Koksma,Szusz}]\label{ETK}
	For all $N,H\in\mathbb{N}$ and $x_1,\ldots,x_N\in\mathbb{R}$, 
	\[
	D(x_1,\ldots,x_N)
	\ll \frac{1}{H} + \sum_{1\le h\le H} \frac{1}{h}\abs{\frac{1}{N}\sum_{n=1}^N e(hx_n)}.
	\]
\end{lemma}

\begin{remark}
	Lemma~\ref{ETK} still holds even if $H$ is a (non-integral) positive number.
	Indeed, the case $0<H<2$ is trivial, since $D(x_1,\ldots,x_N) \le 1 \ll H^{-1}$.
	If $H\ge2$, then the desired inequality follows from 
	\[
	D(x_1,\ldots,x_N)
	\ll \frac{1}{\lfloor{H}\rfloor} + \sum_{1\le h\le\lfloor{H}\rfloor} \frac{1}{h}\abs{\frac{1}{N}\sum_{n=1}^N e(hx_n)}
	\]
	and $H\ge\lfloor{H}\rfloor>H-1\ge H/2$.
	Hence, we use Lemma~\ref{ETK} with not necessarily integral $H>0$.
\end{remark}

The above inequality is sometimes referred as the Erd\H{o}s-Tur\'an-Koksma inequality.
By Lemma~\ref{ETK}, the estimation of discrepancies reduces to that of exponential sums.
In order to apply Lemmas~\ref{2ndderiv} and \ref{3rdderiv}, 
we examine the second and third derivatives of the function $f_r^{-1}(d)$ of $r$, 
where $f_r$ is defined in Proposition~\ref{limsup}.

\begin{lemma}\label{lemU0}
	Let $1<\alpha<2$, $\beta=1/(\alpha-1)$, $d\in\mathbb{N}$ and $c>0$.
	Define the function $y$ of $0<r\le d^{1/\alpha}$ as $y=f_r^{-1}(d)$.
	Then $y\le y+r\le(1+c^{-1})y$ for all $0<r\le d^{1/\alpha}/f_1(c)^{1/\alpha}$.
\end{lemma}
\begin{proof}
	Let $0<r\le d^{1/\alpha}/f_1(c)^{1/\alpha}$.
	Then $f_r(cr)=f_1(c)r^\alpha\le d=f_r(y)$.
	Thus, $cr\le y$ and $y\le y+r\le(1+c^{-1})y$.
\end{proof}

\begin{lemma}\label{lemU2}
	Let $1<\alpha<2$, $\beta=1/(\alpha-1)$ and $d\in\mathbb{N}$.
	Define the function $y$ of $0<r\le d^{1/\alpha}$ as $y=f_r^{-1}(d)$.
	Then 
	\begin{equation}
		y'' = \frac{d(\alpha-1)}{((y+r)^{\alpha-1}-y^{\alpha-1})^3 y^{2-\alpha}(y+r)^{2-\alpha}}. \label{eqU15}
	\end{equation}
	In particular, $y'' \asymp_\alpha d^\beta/r^{\beta+2}$ for all $0<r\le d^{1/\alpha}/2$.
\end{lemma}
\begin{proof}
	Differentiating both sides of $f_r(y)=d$ with respect to $r$, 
	we obtain 
	\[
	(y'+1)(y+r)^{\alpha-1} - y'y^{\alpha-1} = 0.
	\]
	Thus, 
	\begin{equation}
		-y' = \frac{(y+r)^{\alpha-1}}{(y+r)^{\alpha-1}-y^{\alpha-1}} \quad\text{and}\quad
		-y'-1 = \frac{y^{\alpha-1}}{(y+r)^{\alpha-1}-y^{\alpha-1}}. \label{eqU16}
	\end{equation}
	Taking the logarithmic derivative of both sides of the first equality in \eqref{eqU16}, 
	we have 
	\begin{align*}
		\frac{y''}{y'} &= (\alpha-1)\Bigl( \frac{y'+1}{y+r}
		- \frac{(y'+1)(y+r)^{\alpha-2}-y'y^{\alpha-2}}{(y+r)^{\alpha-1}-y^{\alpha-1}} \Bigr)\\
		&= (\alpha-1)\cdot\frac{-(\bcancel{y'}+1)y^{\alpha-1} + y'(\bcancel{y}+r)y^{\alpha-2}}{(y+r)\bigl( (y+r)^{\alpha-1}-y^{\alpha-1} \bigr)}\\
		&= (\alpha-1)y^{\alpha-2}\cdot\frac{-y+ry'}{(y+r)\bigl( (y+r)^{\alpha-1}-y^{\alpha-1} \bigr)}.
	\end{align*}
	Since 
	\begin{align*}
		y-ry' &= \frac{y\bigl( (y+r)^{\alpha-1}-y^{\alpha-1} \bigr) + r(y+r)^{\alpha-1}}{(y+r)^{\alpha-1}-y^{\alpha-1}}\\
		&= \frac{(y+r)^\alpha-y^\alpha}{(y+r)^{\alpha-1}-y^{\alpha-1}}
		= \frac{d}{(y+r)^{\alpha-1}-y^{\alpha-1}},
	\end{align*}
	it follows that 
	\[
	\frac{y''}{y'} = \frac{-d(\alpha-1)y^{\alpha-2}}{(y+r)\bigl( (y+r)^{\alpha-1}-y^{\alpha-1} \bigr)^2}
	\]
	and 
	\[
	y'' = \frac{d(\alpha-1)y^{\alpha-2}(y+r)^{\alpha-2}}{\bigl( (y+r)^{\alpha-1}-y^{\alpha-1} \bigr)^3}.
	\]
	\par
	Next, assume $0<r\le d^{1/\alpha}/2$.
	Then $f_1(1)=2^\alpha-1<2^\alpha$ and $r \le d^{1/\alpha}/2 < d^{1/\alpha}/f_1(1)^{1/\alpha}$.
	By Lemma~\ref{lemU0}, the inequality $y\le y+r\le2y$ holds.
	Also, the mean value theorem implies that 
	\[
	r\alpha y^{\alpha-1} < f_r(y) < r\alpha(y+r)^{\alpha-1}.
	\]
	Since $f_r(y)=d$, it follows that $(d/r\alpha)^\beta<y+r$ and $y<(d/r\alpha)^\beta$.
	Therefore, 
	\begin{equation}
		y \asymp y+r \asymp_\alpha (d/r)^\beta \quad\text{and}\quad
		(y+r)^{\alpha-1}-y^{\alpha-1} \asymp_\alpha ry^{\alpha-2}, \label{eqU19}
	\end{equation}
	whence $y'' \asymp_\alpha d^\beta/r^{\beta+2}$.
\end{proof}

\begin{lemma}\label{lemU3}
	Let $1<\alpha<2$, $\beta=1/(\alpha-1)$ and $d\in\mathbb{N}$.
	Define the function $y$ of $0<r\le d^{1/\alpha}$ as $y=f_r^{-1}(d)$.
	Then 
	\begin{equation}
		-\frac{y'''}{y''} = \frac{(2\alpha-1)ry^{\alpha-1}(y+r)^{\alpha-1} - (2-\alpha)\bigl( (y+r)^{2\alpha-1}-y^{2\alpha-1} \bigr)}
		{\bigl( (y+r)^{\alpha-1}-y^{\alpha-1} \bigr)^2y(y+r)}. \label{eqU18}
	\end{equation}
	In particular, $-y''' \asymp_\alpha d^\beta/r^{\beta+3}$ for all $0<r\le d^{1/\alpha}/2$.
\end{lemma}
\begin{proof}
	Taking the logarithmic derivative of both sides in \eqref{eqU15}, 
	we have 
	\begin{align*}
		\frac{y'''}{y''}
		&= -3(\alpha-1)\frac{(y'+1)(y+r)^{\alpha-2}-y'y^{\alpha-2}}{(y+r)^{\alpha-1}-y^{\alpha-1}}\\
		&\quad- (2-\alpha)\Bigl( \frac{y'}{y} + \frac{y'+1}{y+r} \Bigr).
	\end{align*}
	By \eqref{eqU16}, it follows that 
	\begin{equation}
	\begin{split}
		\frac{y'''}{y''} &= 3(\alpha-1)\frac{y^{\alpha-1}(y+r)^{\alpha-2}-(y+r)^{\alpha-1}y^{\alpha-2}}{\bigl( (y+r)^{\alpha-1}-y^{\alpha-1} \bigr)^2}\\
		&\quad+ \frac{2-\alpha}{(y+r)^{\alpha-1}-y^{\alpha-1}}\Bigl( \frac{(y+r)^{\alpha-1}}{y} + \frac{y^{\alpha-1}}{y+r} \Bigr).
	\end{split}\label{eqU17}
	\end{equation}
	The first term of the right-hand side in \eqref{eqU17} is equal to 
	\begin{align*}
		&\quad 3(\alpha-1)\frac{-ry^{\alpha-2}(y+r)^{\alpha-2}}{\bigl( (y+r)^{\alpha-1}-y^{\alpha-1} \bigr)^2}\\
		&= 3(\alpha-1)\frac{-ry^{\alpha-1}(y+r)^{\alpha-1}}{\bigl( (y+r)^{\alpha-1}-y^{\alpha-1} \bigr)^2y(y+r)},
	\end{align*}
	and the second term of the right-hand side in \eqref{eqU17} is equal to 
	\begin{align*}
		&\quad (2-\alpha)\frac{\bigl( (y+r)^{\alpha-1}-y^{\alpha-1} \bigr)\bigl( (y+r)^\alpha + y^\alpha \bigr)}{\bigl( (y+r)^{\alpha-1}-y^{\alpha-1} \bigr)^2y(y+r)}\\
		&= (2-\alpha)\frac{(y+r)^{2\alpha-1}-y^{2\alpha-1} - ry^{\alpha-1}(y+r)^{\alpha-1}}{\bigl( (y+r)^{\alpha-1}-y^{\alpha-1} \bigr)^2y(y+r)}.
	\end{align*}
	Thus, 
	\begin{align*}
		&\quad \frac{y'''}{y''}\bigl( (y+r)^{\alpha-1}-y^{\alpha-1} \bigr)^2y(y+r)\\
		&= -3(\alpha-1)ry^{\alpha-1}(y+r)^{\alpha-1}\\
		&\quad+ (2-\alpha)\bigl( (y+r)^{2\alpha-1}-y^{2\alpha-1} - ry^{\alpha-1}(y+r)^{\alpha-1} \bigr)\\
		&= -(2\alpha-1)ry^{\alpha-1}(y+r)^{\alpha-1}
		+ (2-\alpha)\bigl( (y+r)^{2\alpha-1}-y^{2\alpha-1} \bigr),
	\end{align*}
	which implies \eqref{eqU18}.
	\par
	Next, assume $0<r\le d^{1/\alpha}/2$.
	Then \eqref{eqU19} holds.
	By the mean value theorem, the numerator of the right-hand side in \eqref{eqU18} is equal to 
	\begin{align*}
		&\quad (2\alpha-1)ry^{\alpha-1}(y+r)^{\alpha-1} - (2-\alpha)\bigl( (y+r)^{2\alpha-1}-y^{2\alpha-1} \bigr)\\
		&= (2\alpha-1)r\bigl( y^{\alpha-1}(y+r)^{\alpha-1} - (2-\alpha)(y+r\theta)^{2\alpha-2} \bigr),
	\end{align*}
	where $\theta=\theta(y,r)\in(0,1)$.
	Since the inequality $y+r\le2y$ holds by Lemma~\ref{lemU0}, 
	it turns out that 
	\begin{align*}
		&\quad y^{\alpha-1}(y+r)^{\alpha-1}
		> y^{\alpha-1}(y+r)^{\alpha-1} - (2-\alpha)(y+r\theta)^{2\alpha-2}\\
		&> 2^{1-\alpha}(y+r)^{2\alpha-2} - (2-\alpha)(y+r)^{2\alpha-2}
		= (2^{1-\alpha}+\alpha-2)(y+r)^{2\alpha-2}.
	\end{align*}
	Noting the inequality $2^{1-\alpha}+\alpha-2>0$, 
	we have 
	\[
	y^{\alpha-1}(y+r)^{\alpha-1} - (2-\alpha)(y+r\theta)^{2\alpha-2} \asymp_\alpha y^{2\alpha-2}.
	\]
	Thus, the numerator of the right-hand side in \eqref{eqU18} is $\asymp_\alpha ry^{2\alpha-2}$.
	This, \eqref{eqU18} and \eqref{eqU19} yield that $-y''' \asymp_\alpha y''/r$.
	Finally, Lemma~\ref{lemU2} implies that $-y''' \asymp_\alpha d^\beta/r^{\beta+3}$.
\end{proof}

\section{Proofs of Lemmas~$\ref{lemU4}$--$\ref{lemU5}$}\label{proof2}

We prove Lemmas~\ref{lemU4}--\ref{lemU5} by using Lemmas~\ref{2ndderiv}, \ref{3rdderiv}, \ref{ETK}, \ref{lemU2} and \ref{lemU3}.

\begin{proof}[Proof of Lemma~$\ref{lemU4}$]
	Note that (i) the inequality $(3-2\beta)/(\beta-1)<\beta-1$ is equivalent to $\alpha<1+1/\sqrt{2}\approx1.707$ if $1<\alpha<2$; 
	(ii) the inequality $(3-2\beta)/(\beta-1)<1/\alpha$ is equivalent to $\alpha<(\sqrt{10}+2)/3\approx1.721$ if $1<\alpha<2$, 
	since 
	\begin{align*}
		&\quad (3-2\beta)/(\beta-1) < 1/\alpha
		\iff \frac{3\alpha-5}{2-\alpha} = \frac{3(\alpha-1)-2}{1-(\alpha-1)} < \frac{1}{\alpha}\\
		&\iff 3\alpha^2-4\alpha-2 < 0
		\iff \alpha < (\sqrt{10}+2)/3.
	\end{align*}
	Now, we can take a positive number $\gamma_1$ with $(3-2\beta)/(\beta-1)<\gamma_1<\min\{ 1/\alpha, \beta-1 \}$ due to the assumption $1<\alpha<1+1/\sqrt{2}$. 
	Then, for sufficiently large $d\ge1$, 
	\begin{align}
		E_1(d)
		&= \#\bigl\{ r\le d^{1/\alpha}/4 : \{f_r^{-1}(d)\} + C_1d^{\beta-1}/r^\beta > 1 \bigr\}\nonumber\\
		&\le d^{\gamma_1} + \sum_{j=\lfloor{\gamma_1\log_2 d}\rfloor}^{\lfloor{(1/\alpha)\log_2 d}\rfloor-2}
		\#\bigl\{ 2^j<r\le2^{j+1} : \{f_r^{-1}(d)\} + C_1d^{\beta-1}/r^\beta > 1 \bigr\}. \label{eqU02}
	\end{align}
	Denote by $D_1(d,R)$ the discrepancy of the sequence $(f_r^{-1}(d))_{R<r\le2R}$.
	The second term of \eqref{eqU02} is bounded from above by 
	\begin{align}
		&\quad \sum_{j=\lfloor{\gamma_1\log_2 d}\rfloor}^{\lfloor{(1/\alpha)\log_2 d}\rfloor-2}
		\#\bigl\{ 2^j<r\le2^{j+1} : \{f_r^{-1}(d)\} + C_1d^{\beta-1}/2^{j\beta} > 1 \bigr\}\nonumber\\
		&\le \sum_{j=\lfloor{\gamma_1\log_2 d}\rfloor}^{\lfloor{(1/\alpha)\log_2 d}\rfloor-2}
		\bigl( C_1d^{\beta-1}/2^{j\beta} + D_1(d,2^j) \bigr)\cdot2^j\nonumber\\
		&= \sum_{j=\lfloor{\gamma_1\log_2 d}\rfloor}^{\lfloor{(1/\alpha)\log_2 d}\rfloor-2}
		\bigl( C_1d^{\beta-1}/2^{j(\beta-1)} + D_1(d,2^j)\cdot2^j \bigr)\nonumber\\
		&\ll_\alpha d^{(1-\gamma_1)(\beta-1)}
		+ \sum_{j=\lfloor{\gamma_1\log_2 d}\rfloor}^{\lfloor{(1/\alpha)\log_2 d}\rfloor-2} D_1(d,2^j)\cdot2^j.
		\label{eqU03}
	\end{align}
	\par
	Now, we estimate the discrepancy $D_1(d,R)$ 
	when $R$ is an integer with $d^{\gamma_1}/2<R\le d^{1/\alpha}/4$.
	Set $H=d^{-\beta/3}R^{(\beta+2)/3}$.
	By Lemma~\ref{ETK}, 
	\[
	D_1(d,R)R \ll H^{-1}R + \sum_{1\le h\le H} \frac{1}{h}\abs{\sum_{R<r\le2R} e(hf_r^{-1}(d))}.
	\]
	Since $2R\le d^{1/\alpha}/2$, Lemmas~\ref{2ndderiv} and \ref{lemU2} imply that 
	\begin{align*}
		\sum_{R<r\le2R} e(hf_r^{-1}(d))
		&\ll_\alpha R\bigl( hd^\beta/R^{\beta+2} \bigr)^{1/2} + \bigl( hd^\beta/R^{\beta+2} \bigr)^{-1/2}\\
		&= h^{1/2}d^{\beta/2}R^{-\beta/2} + h^{-1/2}d^{-\beta/2}R^{(\beta+2)/2}.
	\end{align*}
	Thus, 
	\begin{align*}
		D_1(d,R)R
		&\ll_\alpha H^{-1}R
		+ \sum_{1\le h\le H} \bigl( h^{1/2-1}d^{\beta/2}R^{-\beta/2} + h^{-1/2-1}d^{-\beta/2}R^{(\beta+2)/2} \bigr)\\
		&\ll H^{-1}R + H^{1/2}d^{\beta/2}R^{-\beta/2} + d^{-\beta/2}R^{(\beta+2)/2}\\
		&\ll d^{\beta/3}R^{(1-\beta)/3} + d^{-\beta/2}R^{(\beta+2)/2}.
	\end{align*}
	The second term of \eqref{eqU03} is bounded from above as follows: 
	\begin{equation}
	\begin{split}
		&\quad \sum_{j=\lfloor{\gamma_1\log_2 d}\rfloor}^{\lfloor{(1/\alpha)\log_2 d}\rfloor-2} D_1(d,2^j)\cdot2^j\\
		&\ll_\alpha \sum_{j=\lfloor{\gamma_1\log_2 d}\rfloor}^{\lfloor{(1/\alpha)\log_2 d}\rfloor-2}
		\bigl( d^{\beta/3}\cdot2^{j(1-\beta)/3} + d^{-\beta/2}\cdot2^{j(\beta+2)/2} \bigr)\\
		&\ll_\alpha d^{\beta/3}\cdot d^{\gamma_1(1-\beta)/3} + d^{-\beta/2}\cdot d^{(1/\alpha)(\beta+2)/2}\\
		&= d^{\beta/3+\gamma_1(1-\beta)/3} + d^{1/2\alpha}. 
	\end{split}\label{eqU04}
	\end{equation}
	Since $\gamma_1>(3-2\beta)/(\beta-1)$, it turns out that 
	\[
	\beta/3 - \gamma_1(\beta-1)/3
	< \beta/3 - (3-2\beta)/3
	= \beta-1,
	\]
	whence the first term of \eqref{eqU04} is equal to $o(d^{\beta-1})$.
	Also, noting the equality $\alpha\beta=\beta+1$, 
	we have the following equivalence if $\alpha>1$: 
	\begin{equation}
		1/2\alpha < \beta-1
		\iff 1/2 < \beta+1-\alpha
		\iff -1/2 < \beta-\alpha. \label{eqU05}
	\end{equation}
	From the assumption $1<\alpha<1+1/\sqrt{2}$, it follows that 
	\begin{equation*}
		\beta-\alpha > \sqrt{2} - (1+1/\sqrt{2})
		= \sqrt{2}/2 - 1 > 1/2-1 = -1/2.
	\end{equation*}
	Thus, the second term of \eqref{eqU04} is equal to $o(d^{\beta-1})$.
	Moreover, \eqref{eqU03} and \eqref{eqU02} are also equal to $o(d^{\beta-1})$.
	Therefore, $E_1(d) = o(d^{\beta-1})$.
\end{proof}

\begin{proof}[Proof of Lemma~$\ref{lemU4'}$]
	Note that the inequality $1<\beta\le\sqrt{2}$ is equivalent to $1+1/\sqrt{2}\le\alpha<2$.
	Also, if $5/4<\alpha<2$ (i.e., $1<\beta<4$), then the following equivalences hold: 
	\begin{gather*}
		\beta-1 \le \frac{3-2\beta}{\beta-1} \iff \alpha \ge 1+1/\sqrt{2};\\
		\frac{3-2\beta}{\beta-1} < \frac{6\beta-7}{4-\beta} \iff \alpha < \frac{\sqrt{21}+4}{5} \approx 1.717;\\
		\frac{3-2\beta}{\beta-1} < \frac{4\beta-3}{\beta+3} \iff \alpha < \frac{\sqrt{10}+2}{3} \approx 1.721;\\
		\frac{3\alpha-5}{2-\alpha} = \frac{3-2\beta}{\beta-1} < 1/\alpha \iff \alpha < \frac{\sqrt{10}+2}{3}. 
	\end{gather*}
	Now, we can take positive numbers $\gamma_1$ and $\gamma_2$ with 
	\begin{equation}
		0 < \gamma_2 < \beta-1
		\le \frac{3-2\beta}{\beta-1} < \gamma_1 < \min\Bigl\{ \frac{6\beta-7}{4-\beta}, \frac{4\beta-3}{\beta+3}, 1/\alpha \Bigr\}
		\label{eqU06}
	\end{equation}
	due to the assumption $1+1/\sqrt{2}\le\alpha<(\sqrt{21}+4)/5$.
	Then, for sufficiently large $d\ge1$, 
	\begin{align*}
		E_1(d)
		&= \#\bigl\{ r\le d^{1/\alpha}/4 : \{f_r^{-1}(d)\} + C_1d^{\beta-1}/r^\beta > 1 \bigr\}\\
		&\le d^{\gamma_2} + \sum_{j=\lfloor{\gamma_2\log_2 d}\rfloor}^{\lfloor{\gamma_1\log_2 d}\rfloor-1}
		\#\bigl\{ 2^j<r\le2^{j+1} : \{f_r^{-1}(d)\} + C_1d^{\beta-1}/r^\beta > 1 \bigr\}\\
		&\quad+ \sum_{j=\lfloor{\gamma_1\log_2 d}\rfloor}^{\lfloor{(1/\alpha)\log_2 d}\rfloor-2}
		\#\bigl\{ 2^j<r\le2^{j+1} : \{f_r^{-1}(d)\} + C_1d^{\beta-1}/r^\beta > 1 \bigr\}.
	\end{align*}
	In the same way as the proof of Lemma~\ref{lemU4}, it follows that 
	\begin{equation}
	\begin{split}
		E_1(d) &\ll_\alpha d^{\gamma_2} + d^{(1-\gamma_2)(\beta-1)} + d^{(1-\gamma_1)(\beta-1)}\\
		&\quad+ \sum_{j=\lfloor{\gamma_2\log_2 d}\rfloor}^{\lfloor{\gamma_1\log_2 d}\rfloor-1} D_1(d,2^j)\cdot2^j
		+ \sum_{j=\lfloor{\gamma_1\log_2 d}\rfloor}^{\lfloor{(1/\alpha)\log_2 d}\rfloor-2} D_1(d,2^j)\cdot2^j.
	\end{split}\label{eqU07}
	\end{equation}
	\par
	Now, we estimate the discrepancy $D_1(d,R)$ 
	when $R$ is an integer with $d^{\gamma_2}/2<R\le d^{\gamma_1}/2<d^{1/\alpha}/4$.
	Set $H=d^{-\beta/7}R^{(\beta+3)/7}$.
	Since $2R\le d^{1/\alpha}/2$, Lemmas~\ref{ETK}, \ref{3rdderiv} and \ref{lemU3} imply that 
	\begin{align*}
		D_1(d,R)R &\ll_\alpha H^{-1}R + H^{1/6}R\cdot\bigl( d^\beta/R^{\beta+3} \bigr)^{1/6}
		+ \bigl( d^\beta/R^{\beta+3} \bigr)^{-1/3}\\
		&\ll d^{\beta/7}R^{(4-\beta)/7} + d^{-\beta/3}R^{(\beta+3)/3}.
	\end{align*}
	Noting the inequality $1<\beta\le\sqrt{2}<4$, we obtain 
	\begin{equation}
	\begin{split}
		&\quad \sum_{j=\lfloor{\gamma_2\log_2 d}\rfloor}^{\lfloor{\gamma_1\log_2 d}\rfloor-1} D_1(d,2^j)\cdot2^j\\
		&\ll_\alpha \sum_{j=\lfloor{\gamma_2\log_2 d}\rfloor}^{\lfloor{\gamma_1\log_2 d}\rfloor-1}
		\bigl( d^{\beta/7}\cdot2^{j(4-\beta)/7} + d^{-\beta/3}\cdot2^{j(\beta+3)/3} \bigr)\\
		&\ll_\alpha d^{\beta/7}\cdot d^{\gamma_1(4-\beta)/7} + d^{-\beta/3}\cdot d^{\gamma_1(\beta+3)/3}.
	\end{split}\label{eqU08}
	\end{equation}
	Also, the inequality \eqref{eqU04} holds in the same way as the proof of Lemma~\ref{lemU4}.
	Thus, the inequalities \eqref{eqU04}, \eqref{eqU07} and \eqref{eqU08} yield that 
	\begin{equation}
	\begin{split}
		E_1(d) &\ll_\alpha d^{\gamma_2} + d^{(1-\gamma_2)(\beta-1)} + d^{(1-\gamma_1)(\beta-1)}\\
		&\quad+ d^{\beta/7+\gamma_1(4-\beta)/7} + d^{-\beta/3+\gamma_1(\beta+3)/3}\\
		&\quad+ d^{\beta/3 - \gamma_1(\beta-1)/3} + d^{1/2\alpha}.
	\end{split}\label{eqU09}
	\end{equation}
	By \eqref{eqU06}, all terms of \eqref{eqU09} except for the last term are equal to $o(d^{\beta-1})$.
	Also, noting the inequality $\alpha<(\sqrt{21}+4)/5<(\sqrt{10}+2)/3$, 
	we have 
	\begin{equation}
	\begin{split}
		\beta-\alpha &> 3/(\sqrt{10}-1) - (\sqrt{10}+2)/3\\
		&= (\sqrt{10}+1)/3 - (\sqrt{10}+2)/3\\
		&= -1/3 > -1/2.
	\end{split}\label{eqU10}
	\end{equation}
	By \eqref{eqU05}, the last term of \eqref{eqU09} is equal to $o(d^{\beta-1})$.
	Therefore, $E_1(d)=o(d^{\beta-1})$.
\end{proof}

\begin{proof}[Proof of Lemma~$\ref{lemU5}$]
	Denote by $D_2(d,N)$ the discrepancy of the sequence $( (n^\alpha+d)^{1/\alpha} )_{N\le n<2N}$.
	Then 
	\begin{align}
		E_2(d) &= \#\bigl\{ n<C_2d^{1/\alpha} : \{(n^\alpha+d)^{1/\alpha}\} + 2d^{1/\alpha-1} > 1 \bigr\}\nonumber\\
		&\le \sum_{j=0}^{\lfloor{\log_2(C_2d^{1/\alpha})}\rfloor}
		\#\bigl\{ 2^j\le n<2^{j+1} : \{(n^\alpha+d)^{1/\alpha}\} + 2d^{1/\alpha-1} > 1 \bigr\}\nonumber\\
		&\le \sum_{j=0}^{\lfloor{\log_2(C_2d^{1/\alpha})}\rfloor} \bigl( 2d^{1/\alpha-1}\cdot2^j + D_2(d,2^j)\cdot2^j \bigr)\nonumber\\\
		&\ll_\alpha d^{1/\alpha-1}\cdot d^{1/\alpha} + \sum_{j=0}^{\lfloor{\log_2(C_2d^{1/\alpha})}\rfloor} D_2(d,2^j)\cdot2^j. \label{eqU11}
	\end{align}
	We estimate the discrepancy $D_2(d,N)$ 
	when $d$ and $N$ are positive integers with $1\le N\le C_2d^{1/\alpha}$.
	Set $H=d^{(\alpha-1)/3\alpha}N^{(2-\alpha)/3}$.
	By Lemma~\ref{ETK}, 
	\[
	D_2(d,N)N \ll H^{-1}N + \sum_{1\le h\le H} \frac{1}{h}\abs{\sum_{N\le n<2N} e(h(n^\alpha+d)^{1/\alpha})}.
	\]
	The second derivative of the function $y=(x^\alpha+d)^{1/\alpha}$ is equal to 
	\[
	y''=(\alpha-1)dx^{\alpha-2}(x^\alpha+d)^{1/\alpha-2},
	\]
	which satisfies that $y'' \asymp_\alpha d^{1/\alpha-1}N^{\alpha-2}$ for all $N\le x<2N$ (because of $0<x^d\ll d$).
	By Lemma~\ref{2ndderiv}, 
	\begin{align*}
		&\quad \sum_{N\le n<2N} e(h(n^\alpha+d)^{1/\alpha})\\
		&\ll_\alpha N\bigl( hd^{1/\alpha-1}N^{\alpha-2} \bigr)^{1/2} + \bigl( hd^{1/\alpha-1}N^{\alpha-2} \bigr)^{-1/2}\\
		&= h^{1/2}d^{(1-\alpha)/2\alpha}N^{\alpha/2} + h^{-1/2}d^{(\alpha-1)/2\alpha}N^{(2-\alpha)/2}.
	\end{align*}
	Thus, 
	\begin{align*}
		D_2(d,N)N &\ll_\alpha H^{-1}N + H^{1/2}d^{(1-\alpha)/2\alpha}N^{\alpha/2} + d^{(\alpha-1)/2\alpha}N^{(2-\alpha)/2}\\
		&\ll d^{(1-\alpha)/3\alpha}N^{(\alpha+1)/3} + d^{(\alpha-1)/2\alpha}N^{(2-\alpha)/2},
	\end{align*}
	and moreover, the second term of \eqref{eqU11} is bounded from above as follows: 
	\begin{align*}
		&\quad \sum_{j=0}^{\lfloor{\log_2(C_2d^{1/\alpha})}\rfloor} D_2(d,2^j)\cdot2^j\\
		&\ll_\alpha \sum_{j=0}^{\lfloor{\log_2(C_2d^{1/\alpha})}\rfloor}
		\bigl( d^{(1-\alpha)/3\alpha}\cdot2^{j(\alpha+1)/3} + d^{(\alpha-1)/2\alpha}\cdot2^{j(2-\alpha)/2} \bigr)\\
		&\ll_\alpha d^{(1-\alpha)/3\alpha}\cdot d^{(\alpha+1)/3\alpha} + d^{(\alpha-1)/2\alpha}\cdot d^{(2-\alpha)/2\alpha}\\
		&= d^{2/3\alpha} + d^{1/2\alpha} \ll d^{2/3\alpha}.
	\end{align*}
	Also, noting the equality $\alpha\beta=\beta+1$, 
	we have the following equivalence if $\alpha>1$: 
	\begin{equation*}
		2/3\alpha < \beta-1
		\iff 2/3 < \beta+1-\alpha
		\iff -1/3 < \beta-\alpha.
	\end{equation*}
	Since \eqref{eqU10} follows from the assumption $1<\alpha<(\sqrt{10}+2)/3$, 
	the second term of \eqref{eqU11} is equal to $o(d^{\beta-1})$.
	Since the first term of \eqref{eqU11} is also equal to $o(d^{\beta-1})$, 
	we complete the proof.
\end{proof}

\section{Proof of Theorem~$\ref{main2}$}\label{AE}

We prove Proposition~\ref{UB}, Lemmas~\ref{lemU6}, \ref{lemU7} and Theorem~\ref{main2}.

\begin{proof}[Proof of Proposition~$\ref{UB}$]
	The number $\cN_\alpha^{\ge l}(d)$ is equal to $E_0(d,l-1)$ defined in \eqref{eqU20}.
	Thus, by \eqref{eqU12}, we have 
	\begin{align*}
		\cN_\alpha^{\ge l}(d)
		&\le \#\{ (n,r)\in\mathbb{N}^2 : l\le r\le(d-1)^{1/\alpha}/4,\ d-1 < f_r(n) < d+1 \}\\
		&\quad+ \#\{ (n,r)\in\mathbb{N}^2 : r>(d-1)^{1/\alpha}/4,\ d-1 < f_r(n) < d+1 \}.
	\end{align*}
	The first term of the right-hand side is equal to $O_\alpha(d^{\beta-1}\sum_{r\ge l} 1/r^\beta) + E_1^{\ge l}(d-1)$, and 
	the second term of the right-hand side is equal to $O_\alpha(d^{2/\alpha-1}) + E_2(d-1)$ 
	in the same way as Steps~3 and 4 of the proof of Proposition~\ref{limsup}.
	Since $\sum_{r\ge l} 1/r^\beta \ll_\alpha l^{1-\beta}$, we obtain the desired inequality.
\end{proof}

\begin{proof}[Proof of Lemma~$\ref{lemU6}$]
	By the definition of $E_1^{\ge l}(d)$, 
	\begin{align*}
		E_1^{\ge l}(d)
		&= \#\bigl\{ r\in\mathbb{N} : l\le r\le d^{1/\alpha}/4,\ \{f_r^{-1}(d)\} + C_1d^{\beta-1}/r^\beta > 1 \bigr\}\\
		&\le 1 + \sum_{j=\lfloor{\log_2 l}\rfloor}^{\lfloor{(1/\alpha)\log_2 d}\rfloor-2}
		\#\bigl\{ 2^j<r\le2^{j+1} : \{f_r^{-1}(d)\} + C_1d^{\beta-1}/r^\beta > 1 \bigr\}.
	\end{align*}
	Thus, 
	\[
	E_1^{\ge l}(d)
	\ll_\alpha d^{\beta-1}l^{1-\beta} + d^{\beta/3}l^{(1-\beta)/3} + d^{1/2\alpha}
	\]
	in the same way as the proof of Lemmas~\ref{lemU4} (replace $d^{\gamma_1}$ with $l$).
	Since the inequality $1/2\alpha\le2/\alpha-1$ follows from the assumption $1<\alpha\le3/2$, 
	we obtain the desired inequality.
\end{proof}

\begin{proof}[Proof of Lemma~$\ref{lemU7}$]
	By the proof of Lemmas~\ref{lemU5}, 
	\[
	E_2(d) \ll_\alpha d^{2/\alpha-1} + d^{2/3\alpha}.
	\]
	Since the inequality $2/3\alpha\le2/\alpha-1$ follows from the assumption $1<\alpha\le4/3$, 
	we obtain the desired inequality.
\end{proof}

\begin{proof}[Proof of Theorem~$\ref{main2}$]
	Let $\alpha>1$ and $N\in\mathbb{N}$.
	First, we show that 
	\begin{equation}
		\cE_\alpha(N) \le N^2 + 6\sum_{l=1}^{N-1} \sum_{(l-1)\alpha N^{\alpha-1}\le d-1<l\alpha N^{\alpha-1}} \cN_\alpha^{\ge l}(d)^2.
		\label{eq21}
	\end{equation}
	Noting that the equation $\lfloor{n_1^\alpha}\rfloor+\lfloor{n_2^\alpha}\rfloor=\lfloor{n_3^\alpha}\rfloor+\lfloor{n_4^\alpha}\rfloor$ 
	is equivalent to 
	\[
	\exists d\in\mathbb{Z},\ 
	\lfloor{n_1^\alpha}\rfloor-\lfloor{n_3^\alpha}\rfloor=d \quad\text{and}\quad
	\lfloor{n_4^\alpha}\rfloor-\lfloor{n_2^\alpha}\rfloor=d,
	\]
	we have 
	\begin{align}
		\cE_\alpha(N)
		&= N^2 + 2\sum_{d=1}^\infty \Biggl( 
		\#\Biggl\{ (n,r)\in\mathbb{N}^2 : 
		\begin{array}{c}
			n+r\le N,\\
			\lfloor{(n+r)^\alpha}\rfloor-\lfloor{n^\alpha}\rfloor=d
		\end{array}
		\Biggr\} \Biggr)^2\nonumber\\
		&= N^2 + 2\sum_{d=1}^\infty \Biggl( \sum_{r=1}^{N-1} 
		\#\Biggl\{ n\in\mathbb{N} : 
		\begin{array}{c}
			n+r\le N,\\
			\lfloor{(n+r)^\alpha}\rfloor-\lfloor{n^\alpha}\rfloor=d
		\end{array}
		\Biggr\} \Biggr)^2. \label{eq22}
	\end{align}
	Set the summands of the inner sum in \eqref{eq22} as 
	\[
	N_{r,d} =
	\#\Biggl\{ n\in\mathbb{N} : 
	\begin{array}{c}
		n+r\le N,\\
		\lfloor{(n+r)^\alpha}\rfloor-\lfloor{n^\alpha}\rfloor=d
	\end{array}
	\Biggr\} \quad (d\ge1,\ 1\le r<N).
	\]
	Then 
	\begin{equation}
	\begin{split}
		\cE_\alpha(N)
		&= N^2 + 2\sum_{d=1}^\infty \Biggl( \sum_{r=1}^{N-1} N_{r,d} \Biggr)^2\\
		&= N^2 + 2\sum_{d=1}^\infty \sum_{r=1}^{N-1} N_{r,d}^2 + 4\sum_{d=1}^\infty \sum_{1\le r_1<r_2<N} N_{r_1,d}N_{r_2,d}.
	\end{split}\label{eq23}
	\end{equation}
	Now, consider the equation $\lfloor{(n+r)^\alpha}\rfloor-\lfloor{n^\alpha}\rfloor=d$ in the variables $n,r\in\mathbb{N}$ when $d\ge1$ is an integer.
	If $n+r\le N$, then the mean value theorem implies that 
	\begin{equation}
		d-1 < (n+r)^\alpha-n^\alpha < r\alpha(n+r)^{\alpha-1} \le r\alpha N^{\alpha-1}.
		\label{eq24}
	\end{equation}
	By this, 
	\begin{equation}
	\begin{split}
		\sum_{d=1}^\infty \sum_{r=1}^{N-1} N_{r,d}^2
		&= \sum_{r=1}^{N-1} \sum_{0\le d-1<r\alpha N^{\alpha-1}} N_{r,d}^2\\
		&\le \sum_{l=1}^{N-1} \sum_{(l-1)\alpha N^{\alpha-1}\le d-1<l\alpha N^{\alpha-1}} \cN_\alpha^{\ge l}(d)^2,
	\end{split}\label{eq25}
	\end{equation}
	where we have used 
	\begin{equation}
	\begin{split}
		\sum_{r=1}^{N-1} \sum_{0\le d-1<r\alpha N^{\alpha-1}}
		&= \sum_{r=1}^{N-1} \sum_{l=1}^r \sum_{(l-1)\alpha N^{\alpha-1}\le d-1<l\alpha N^{\alpha-1}}\\
		&= \sum_{l=1}^{N-1} \sum_{r=l}^{N-1} \sum_{(l-1)\alpha N^{\alpha-1}\le d-1<l\alpha N^{\alpha-1}}
	\end{split}\label{eq26}
	\end{equation}
	to obtain the last inequality.
	Also, 
	\begin{equation}
		\sum_{d=1}^\infty \sum_{1\le r_1<r_2<N} N_{r_1,d}N_{r_2,d}
		= \sum_{2\le r_2<N} \sum_{d=1}^\infty \sum_{1\le r_1<r_2} N_{r_1,d}N_{r_2,d}.
		\label{eq27}
	\end{equation}
	By \eqref{eq26} (when replacing $r$ and $N$ with $r_1$ and $r_2$, respectively) and the fact around \eqref{eq24}, 
	\begin{align*}
		&\quad \sum_{d=1}^\infty \sum_{1\le r_1<r_2} N_{r_1,d}N_{r_2,d}
		= \sum_{1\le r_1<r_2} \sum_{0\le d-1<r_1\alpha N^{\alpha-1}} N_{r_1,d}N_{r_2,d}\\
		&= \sum_{l=1}^{r_2-1} \sum_{r_1=l}^{r_2-1} \sum_{(l-1)\alpha N^{\alpha-1}\le d-1<l\alpha N^{\alpha-1}} N_{r_1,d}N_{r_2,d}\\
		&\le \sum_{1\le l<r_2} \sum_{(l-1)\alpha N^{\alpha-1}\le d-1<l\alpha N^{\alpha-1}} \cN_\alpha^{\ge l}(d)N_{r_2,d}.
	\end{align*}
	By this and \eqref{eq27}, 
	\begin{align*}
		&\quad \sum_{d=1}^\infty \sum_{1\le r_1<r_2<N} N_{r_1,d}N_{r_2,d}\\
		&\le \sum_{2\le r_2<N} \sum_{1\le l<r_2} \sum_{(l-1)\alpha N^{\alpha-1}\le d-1<l\alpha N^{\alpha-1}} \cN_\alpha^{\ge l}(d)N_{r_2,d}\\
		&= \sum_{l=1}^{N-2} \sum_{l<r_2<N} \sum_{(l-1)\alpha N^{\alpha-1}\le d-1<l\alpha N^{\alpha-1}} \cN_\alpha^{\ge l}(d)N_{r_2,d}\\
		&\le \sum_{l=1}^{N-2} \sum_{(l-1)\alpha N^{\alpha-1}\le d-1<l\alpha N^{\alpha-1}} \cN_\alpha^{\ge l}(d)^2.
	\end{align*}
	This, \eqref{eq25} and \eqref{eq23} yield \eqref{eq21}.
	\par
	Next, assuming $1<\alpha\le4/3$, we show that $\cE_\alpha(N) \ll_\alpha N^{4-\alpha}$.
	Proposition~\ref{UB}, Lemmas~\ref{lemU6} and \ref{lemU7} imply that 
	\[
	\cN_\alpha^{\ge l}(d) \ll_\alpha l^{1-\beta}d^{\beta-1} + l^{(1-\beta)/3}d^{\beta/3} + d^{2/\alpha-1}.
	\]
	This and \eqref{eq21} yield that 
	\begin{align*}
		\cE_\alpha(N)
		&\ll N^2 + \sum_{l=1}^{N-1} \sum_{(l-1)\alpha N^{\alpha-1}\le d-1<l\alpha N^{\alpha-1}} \cN_\alpha^{\ge l}(d)^2\\
		&\ll_\alpha N^2 + \sum_{l=1}^{N-1} \sum_{(l-1)\alpha N^{\alpha-1}\le d-1<l\alpha N^{\alpha-1}}
		(l^{2-2\beta}d^{2\beta-2} + l^{2(1-\beta)/3}d^{2\beta/3} + d^{4/\alpha-2}).
	\end{align*}
	The first sum with summands $l^{2-2\beta}d^{2\beta-2}$ is 
	\begin{align*}
		&\quad \sum_{l=1}^{N-1} \sum_{(l-1)\alpha N^{\alpha-1}\le d-1<l\alpha N^{\alpha-1}} l^{2-2\beta}d^{2\beta-2}\\
		&\le \sum_{l=1}^{N-1} \int_{\lceil{(l-1)\alpha N^{\alpha-1}}\rceil+1}^{\lceil{l\alpha N^{\alpha-1}}\rceil+1} l^{2-2\beta}x^{2\beta-2}\,dx\\
		&\ll_\alpha \sum_{l=1}^{N-1} l^{2-2\beta}\cdot N^{\alpha-1}( lN^{\alpha-1} )^{2\beta-2}\\
		&= \sum_{l=1}^{N-1} N^{3-\alpha}
		< N^{4-\alpha}.
	\end{align*}
	The second sum with summands $l^{2(1-\beta)/3}d^{2\beta/3}$ is 
	\begin{align*}
		&\quad \sum_{l=1}^{N-1} \sum_{(l-1)\alpha N^{\alpha-1}\le d-1<l\alpha N^{\alpha-1}} l^{2(1-\beta)/3}d^{2\beta/3}\\
		&\le \sum_{l=1}^{N-1} \int_{\lceil{(l-1)\alpha N^{\alpha-1}}\rceil+1}^{\lceil{l\alpha N^{\alpha-1}}\rceil+1} l^{2(1-\beta)/3}x^{2\beta/3}\,dx\\
		&\ll_\alpha \sum_{l=1}^{N-1} l^{2(1-\beta)/3}\cdot N^{\alpha-1}( lN^{\alpha-1} )^{2\beta/3}\\
		&= \sum_{l=1}^{N-1} l^{2/3}N^{\alpha-1/3}
		< N^{\alpha+4/3}.
	\end{align*}
	The third sum with summands $d^{4/\alpha-2}$ is 
	\begin{align*}
		&\quad \sum_{l=1}^{N-1} \sum_{(l-1)\alpha N^{\alpha-1}\le d-1<l\alpha N^{\alpha-1}} d^{4/\alpha-2}\\
		&\le \sum_{l=1}^{N-1} \int_{\lceil{(l-1)\alpha N^{\alpha-1}}\rceil+1}^{\lceil{l\alpha N^{\alpha-1}}\rceil+1} x^{4/\alpha-2}\,dx
		\ll_\alpha \sum_{l=1}^{N-1} N^{\alpha-1}( lN^{\alpha-1} )^{4/\alpha-2}\\
		&= \sum_{l=1}^{N-1} l^{4/\alpha-2}N^{5-\alpha-4/\alpha}
		< N^{4/\alpha-1}N^{5-\alpha-4/\alpha} = N^{4-\alpha}.
	\end{align*}
	Since the inequality $\alpha+4/3 \le 4-\alpha$ follows from the assumption $1<\alpha\le4/3$, 
	we obtain $\cE_\alpha(N) \ll_\alpha N^{4-\alpha}$.
\end{proof}

\section*{Acknowledgments}
The author thanks Dr.\ Kota Saito for reading the first draft and finding mistakes.
Also, the author is grateful to Prof.\ Christoph Aistleitner for 
suggesting that Theorem~\ref{main0} probably gives a precise upper bound for the additive energies of Piatetski-Shapiro sequences and 
telling Refs.~\cite{Robert-Sargos, Aistleitner1} to the author.
The author was supported by JSPS KAKENHI Grant Numbers JP19J20161, JP22J00339 and JP22KJ1621.

\end{document}